\documentclass{amsart}
\usepackage[utf8]{inputenc}
\usepackage{amscd}
\usepackage{amssymb}
\usepackage{pstricks}
\usepackage{graphicx}
\newtheorem{theorem}{Theorem}
\newtheorem{proposition}{Proposition}

\newtheorem{definition}{Definition}
\newtheorem{corollary}{Corollary}

\newtheorem{remark}{Remark}
\newtheorem{example}{Example}
\DeclareMathOperator{\rank}{Rank}

\DeclareMathOperator{\ima}{im}

\author[Pinto $\&$ Robledo]{{Nicol\'as Pinto} \& {Gonzalo Robledo}}
\email{nicolas.pinto.p@ug.uchile.cl,grobledo@uchile.cl}

\address{Departamento de Matem\'aticas, Universidad de Chile, Casilla 653, Santiago, Chile.}
\keywords{Bohl exponents; Exponential dichotomy spectrum; Difference Equations}

\title[Bohl Exponents and Exponential Dichotomy Spectrum]{Some relations between Bohl Exponents and the Exponential Dichotomy spectrum}

\date{December 2018}

\begin{document}

\begin{abstract}
We study a liaison between the Bohl's exponents and the exponential dichotomy spectrum of a non autonomous linear system of difference equations on the whole line $\mathbb{Z}$. More specifically, We prove that for any initial condition in an invariant vector bundle, associated to its  exponential dichotomy spectrum, its Bohl's exponents are contained in an spectral interval. 
\end{abstract}

\maketitle

\section{Introduction}
The purpose of this note is to study the localization of the Bohl's exponents with respect to the bounds of the spectral intervals of the exponential dichotomy spectrum (formal definitions will be given below) 
in a linear system of difference equations
\begin{equation}
\label{lineal}
x_{n+1}=A(n)x_{n},
\end{equation}
where the sequence of matrices $A\colon \mathbb{Z}\to M_{d}(\mathbb{R})$ satisfies the following properties:
\begin{itemize}
    \item[\textbf{(A1)}] The matrices $A(n)$ are non--singular for any $n\in \mathbb{Z}$.
    \item[\textbf{(A2)}] Given a matrix norm $||\cdot||$, there exists $M>0$ such that 
    $$
    \sup\limits_{n\in \mathbb{Z}}\{||A(n)||,||A^{-1}(n)||\}\leq M \quad \textnormal{for any $n\in \mathbb{Z}$}.
    $$
\end{itemize}

\begin{definition}
\label{transition}
The transition operator of (\ref{lineal}) is the map $X\colon \mathbb{Z}\times \mathbb{Z}\to M_{d}(\mathbb{R})$ defined as follows:
\begin{equation}
\label{EO}
X(m,n)=\left\{\begin{array}{ccl}
A(m-1)\cdots A(n) & \textnormal{if} & m>n \\
I  & \textnormal{if} &   m=n \\
A^{-1}(m)\cdots A^{-1}(n-1) &\textnormal{if} & m<n.
\end{array}\right.
\end{equation}
\end{definition}

The relation between Bohl's exponents  and the exponential dichotomy spectrum of (\ref{lineal}) has been deeply studied on the half line $\mathbb{Z}_0^+$ both in the continuous \cite{Doan} and 
in the discrete case \cite{Potzsche,Potzsche-Russ}. Nevertheless, there are less results when considering the whole line and this note can be seen as a contribution in this context.

The structure of the note is as follows: in the Section 2 we recall some basic facts about the Bohl's exponents and the exponential dichotomy spectrum. While in  Section 3 we introduce our main results and present a brief discussion about them.

\section{Mathematical preliminaries}

 \subsection{Bohl's exponents}
 We will work with the following definition, considered in \cite{Babiarz15,Babiarz17-b,Nieza}:
 \begin{definition}
\label{UBE}
The upper and lower Bohl exponents of a solution $k\mapsto X(k,0)\xi$ of (\ref{lineal}) passing through $\xi\in \mathbb{R}^{d}\setminus\{0\}$ at $k=0$ are respectively   
\begin{equation}
    \label{UBE1}
\overline{\beta}_{A}(\xi):=  \limsup\limits_{m,n-m\to +\infty}\left(\frac{||X(n,0)\xi||}{||X(m,0)\xi||}\right)^{\frac{1}{n-m}}
\end{equation}
and
\begin{equation}
    \label{LBE1}
\underline{\beta}_{A}(\xi):=  \liminf\limits_{m,n-m\to +\infty}\left(\frac{||X(n,0)\xi||}{||X(m,0)\xi||}\right)^{\frac{1}{n-m}}.
\end{equation}
\end{definition}

As pointed out in \cite{Doan}, the Bohl's exponents $\overline{\beta}_{A}(\xi)$ and $\underline{\beta}_{A}(\xi)$ can be seen as measures of the biggest and smallest growth rate of the solution
$k\mapsto X(k,0)\xi$ of (\ref{lineal}). To the best of our knowledge, the Bohl's exponents for discrete systems have been studied firstly by Ben--Artzi and Gohberg \cite{Ben-Artzi} as a corresponding 
version of the continuous case studied in \cite[Ch.III]{Daleckii}, they have also been defined in an alternative but equivalent formulation by P\"otzsche in \cite{Pot2010,Potzsche,Potzsche-Russ} for algebraic and scalar difference equations.

The discrete Bohl's exponent theory has been fashioned along with extensive literature \cite{Barabanov,Daleckii,Doan,Vinograd} developed after the seminal works of Bohl \cite{Bohl} and Persidskii \cite{Persidskii} in the continuous case.
We refer the reader to the works \cite{Babiarz15,Babiarz16,Babiarz17-b,Czornik17,Nieza} for a good review and a detailed description of the relation between discrete and  continuous definitions. In this context,
the authors (see also \cite{Du}) define the senior upper and junior lower general exponents of the system (\ref{lineal}) respectively as follows:
\begin{equation}
\label{Se}
\Omega^{0}(A)=\limsup\limits_{m,n-m\to + \infty}||X(n,m)||^{\frac{1}{n-m}} 
\end{equation}
and
\begin{equation}
\label{JE}
\omega_{0}(A)=\liminf\limits_{m,n-m\to + \infty}||X(n,m)||^{\frac{-1}{n-m}},
\end{equation}
whose continuous version was introduced in Bohl's seminal work.

In addition, it can be shown (see \emph{e.g.} \cite[p.339]{Babiarz15}) that we can neglect the condition $m\to\infty$ in Definition \ref{UBE}. Namely, we have the following characterization for the Bohl exponents:
\begin{equation}
\label{UBE2}
\overline{\beta}_{A}(\xi)=  \limsup\limits_{n-m\to +\infty}\left(\frac{||X(n,0)\xi||}{||X(m,0)\xi||}\right)^{\frac{1}{n-m}}, 
\end{equation}
and
\begin{equation}
\label{LBE2}
\underline{\beta}_{A}(\xi)=  \liminf\limits_{n-m\to +\infty}\left(\frac{||X(n,0)\xi||}{||X(m,0)\xi||}\right)^{\frac{1}{n-m}}.
\end{equation}

In \cite{Babiarz15,Nieza}, the authors study a linear diagonal system (\ref{lineal}) and prove that the number of upper Bohl's exponents cannot be bigger than $2^d-1$.
This result is used in \cite{Babiarz17,Czornik17} in order to characterize the functions which can be the upper and lower Bohl function for a diagonal system.

In \cite{Babiarz17-b}, the authors characterize the senior upper general exponent in terms of upper Bohl's exponents as follows
\begin{displaymath}
\Omega^{0}(A)=\lim\limits_{\varepsilon\to 0^{+}}\sup\limits_{||Q||_{\infty}<\varepsilon}\sup\limits_{\xi\in \mathbb{R}^{d}\setminus \{0\}}\overline{\beta}_{A+Q}(\xi).
\end{displaymath}

\subsection{Exponential dichotomy spectrum}
Exponential dichotomy can be seen as a possible extension of the well known property of hyperbolicity to the nonautonomous framework. In order to define it, we need to introduce the following definition:
\begin{definition}
\label{IP0}
An invariant projector of (\ref{lineal}) is a map $P\colon \mathbb{Z}\to M_{d}(\mathbb{R})$ verifying
\begin{equation}
\label{IP1}
P^{2}(k)=P(k) \quad \textnormal{and} \quad P(k+1)A(k)=A(k)P(k)  \quad \textnormal{for any $k\in \mathbb{Z}$}.
\end{equation}
\end{definition}
It is straightforward to verify that if $P$ is an invariant projector of (\ref{lineal}), then
\begin{equation}
\label{IP2}
X(k,\ell)P(\ell) = P(k)X(k,\ell) \quad \textnormal{for any $k,\ell \in \mathbb{Z}$}.
\end{equation}

\begin{definition}
\label{dichotomy}
The system (\ref{lineal}) has an exponential dichotomy on $\mathbb{Z}$ if there exists an invariant projector $P$ and two constants $K>0,\rho\in (0,1)$ such that
\begin{equation}
\label{ED1}
\left\{\begin{array}{rcl}
||X(k,\ell)P(\ell)||\leq K\rho^{k-\ell} &\textnormal{if} &  k\geq \ell  \quad \textnormal{and}\quad  k,\ell\in \mathbb{Z}\\\\
||X(k,\ell)[I-P(\ell)]||\leq K\rho^{\ell-k} &\textnormal{if} & k\leq \ell \quad \textnormal{and}\quad  k,\ell\in \mathbb{Z}.
\end{array}\right.
\end{equation}

\end{definition}

\begin{definition}
\label{spectrum1}
The exponential dichotomy spectrum of (\ref{lineal}) is the set $\Sigma(A)$ composed of all numbers $\gamma \in (0,+\infty)$ such that 
the system
\begin{equation}   
\label{lineal-g}
    x(n+1)=\frac{1}{\gamma}A(n)x(n) .
\end{equation}    
has no exponential dichotomy on $\mathbb{Z}$. Moreover, the set $\rho(A):=(0,+\infty)\setminus \Sigma(A)$ is called the resolvent of (\ref{lineal}).
\end{definition}

The spectral theory associated to the exponential dichotomy has been extensively developed for discrete and continuous systems both in finite \cite{Aulbach,Potzsche-Russ} and infinite dimensions \cite{Russ}.
In order to contextualize our main result, We will recall a fundamental result:
\begin{proposition}\cite[Th.2.1]{Aulbach}
If \textnormal{\textbf{(A1)--(A2)}} are satisfied, then the exponential dichotomy spectrum of (\ref{lineal}) is the union of $\ell\leq d$ closed intervals as follows:
\begin{equation*}
    \Sigma(A)=[a_{1},b_{1}]\cup [a_{2},b_{2}] \cup \cdots \cup [a_{\ell},b_{\ell}],
\end{equation*}
where $a_{i}\leq b_{i} <a_{i+1}$ for any $i=1,\ldots,\ell-1$ and the resolvent of (\ref{lineal}) is the union of open intervals: 
\begin{displaymath}
\rho(A)=\rho_{0}(A)\cup \rho_{1}(A) \cup \ldots \cup \rho_{\ell}(A) \quad \textnormal{where} \quad \rho_{i}(A)=\left\{\begin{array}{rl}
(0,a_{1}) & \textnormal{if $i=0$}\\
(b_{i},a_{i+1}) & \textnormal{if $1<i<\ell$}\\
(b_{\ell},+\infty) & \textnormal{if $i=\ell$}.
\end{array}\right. 
\end{displaymath}
\end{proposition}

The intervals $[a_{i},b_{i}]$ are known as the $i$th \emph{spectral intervals} of $\Sigma(A)$ while the intervals $\rho_{i}(A)$ are known as the $i$th \emph{spectral gaps} or
$i$th \emph{connected components} of $\rho(A)$.

\begin{remark}
\label{obs1}
It is well known (see \cite{Aulbach} for details) that:\newline

\noindent i) By definition of $\Sigma(A)$, for any $\gamma \in \rho_{j}(A)$ the system (\ref{lineal-g}) has an exponential dichotomy with transition matrix $X_{\gamma}(m,n)=\frac{1}{\gamma} X(m,n)$, an invariant projector $P_{\gamma}$ and constants $K_{\gamma}>0,\rho_{\gamma}\in (0,1)$.

\noindent ii) The images of the above projectors $P_{\gamma}$ coincide for any $\gamma \in \rho_{j}(A)$. 
    
\noindent iii) If $\gamma,\mu\in \rho(A)$ where $\gamma<\mu$ and they are not in the same spectral gap, then $\rank P_{\gamma}<\rank P_{\mu}$. Moreover, if $\gamma\in \rho_{0}(A)$ and 
    $\mu \in \rho_{\ell}(A)$ then
    $P_{\gamma}=0$ and $P_{\mu}=I$.
\end{remark}

For each $i\in\{1,\cdots,\ell\}$, let $\gamma_i\in \rho_i(A)$ and consider the vector bundles
\begin{displaymath}
\mathcal{W}_{i}=\bigcup\limits_{k\in \mathbb{Z}}\{k\}\times \left[\ker P_{\gamma_{i-1}}(k)\cap \ima P_{\gamma_{i}}(k)\right],
\end{displaymath}
which are induced by the fibers
\begin{equation}
\mathcal{W}_{i}(n)=\ker P_{\gamma_{i-1}}(n)\cap \ima P_{\gamma_{i}}(n).
\end{equation}
Moreover, we also define $\mathcal{W}_0:=\mathbb{Z}\times \{0\}$ and $\mathcal{W}_{\ell+1}:=\mathbb{Z}\times\{0\}$.

We will recall a second fundamental result necessary to introduce our main result:
\begin{proposition}\cite[Th. 2.1]{Aulbach}
The sets $\mathcal{W}_{i}$ are invariant vector bundles of (\ref{lineal}) and they are independent of the choice of $\gamma_{i}$ (with $i=0,\dots,\ell$). Moreover
$$
\mathcal{W}_{0}\oplus \dots \oplus \mathcal{W}_{\ell+1}=\mathbb{Z}\times \mathbb{R}^{d}
$$
is a Whitney sum, namely, $\mathcal{W}_{i}\cap\mathcal{W}_{j}=\mathbb{Z}\times \{0\}$ for $i\neq j$ and
$$
\mathcal{W}_{0} + \dots + \mathcal{W}_{\ell+1}=\mathbb{Z}\times \mathbb{R}^{d}.
$$
\end{proposition}

\section{Main results}
\begin{theorem}
\label{T1}
For any initial condition $\xi_{i}\in \mathcal{W}_{i}(0)\setminus \{0\}$ of (\ref{lineal}) with $i=1,\ldots,\ell$, then it follows that $\overline{\beta}_{A}(\xi_{i})\in [a_{i},b_{i}]$.
\end{theorem}

\begin{proof}
As $\xi_{i}\in \mathcal{W}_{i}(0)$ it follows that $\xi_{i} \in  \ima P_{\gamma_{i}}(0) \cap \ker P_{\gamma_{i-1}}(0)$. The proof will be decomposed in two steps.

\noindent \emph{Step 1}: As $\xi_{i} \in  \ima P_{\gamma_{i}}(0)$ and consequently $P_{\gamma_{i}}(0)\xi_{i}=\xi_{i}$, then we can deduce that
\begin{displaymath}
\begin{array}{rcl}
\displaystyle \frac{||X(n,0)\xi_{i}||}{||X(m,0)\xi_{i}||} & = & \displaystyle  \frac{||X(n,0)\gamma_{i}^{-n}\xi_{i}||}{||X(m,0)\gamma_{i}^{-m}\xi_{i}||}\,\gamma_{i}^{n-m} \\\\
&=&  \displaystyle  \frac{||X_{\gamma_{i}}(n,0)\xi_{i}||}{||X_{\gamma_{i}}(m,0)\xi_{i}||}\,\gamma_{i}^{n-m} \\\\

&=&  \displaystyle  \frac{||X_{\gamma_{i}}(n,m)X_{\gamma_{i}}(m,0)\xi_{i}||}{||X_{\gamma_{i}}(m,0)\xi_{i}||}\,\gamma_{i}^{n-m} \\\\
&=&  \displaystyle  \frac{||X_{\gamma_{i}}(n,m)P_{\gamma_{i}}(m)X_{\gamma_{i}}(m,0)\xi_{i}||}{||X_{\gamma_{i}}(m,0)\xi_{i}||}\gamma_{i}^{n-m},
\end{array}
\end{displaymath}
where we have used that $X_{\gamma_{i}}$ is the transition matrix of (\ref{lineal-g}) with $\gamma=\gamma_{i}$ and the identity 
$P_{\gamma_{i}}(m)X_{\gamma_{i}}(m,0)\xi_{i}=X_{\gamma_{i}}(m,0)\xi_{i}$, which follows from (\ref{IP2}) combined with the fact $P_{\gamma_{i}}(0)\xi_{i}=\xi_{i}$.

The above identities combined with the fact that $\gamma_{i}\in \rho_{i}(A)$ imply that for $n\geq m$: 
\begin{displaymath}
\begin{array}{rcl}
\displaystyle \frac{||X(n,0)\xi_{i}||}{||X(m,0)\xi_{i}||} & \leq  & ||X_{\gamma_{i}}(n,m)P_{\gamma_{i}}(m)||\gamma_{i}^{n-m} \\\\
                                                          & \leq  & K_{\gamma_{i}}\rho_{\gamma_{i}}^{n-m}\gamma_{i}^{n-m},
\end{array}
\end{displaymath}
where $K_{\gamma_{i}}$ and $\rho_{\gamma_{i}}$ are the constants defined in the statement i) from Remark \ref{obs1} with $\gamma=\gamma_{i}$.
Now, by using the fact that $0<\rho_{\gamma_{i}}<1$, it is easy to see that
\begin{displaymath}
\overline{\beta}_{A}(\xi_{i})=\limsup\limits_{n-m\to +\infty}\left(\frac{||X(n,0)\xi_{i}||}{||X(m,0)\xi_{i}||}\right)^{\frac{1}{n-m}} \leq \rho_{\gamma_{i}}\gamma_{i}<\gamma_{i}\in \rho_{i}(A).
\end{displaymath}

As the above result is independent of the choice of $\gamma_{i}\in (b_{i},a_{i+1})$, we have that $\overline{\beta}(\xi_{i})<b_{i}+\varepsilon$ for any $\varepsilon>0$ small enough
and we have that 
\begin{equation}
\label{derecha}
\overline{\beta}(\xi_{i})\leq b_{i}.
\end{equation}

\noindent \emph{Step 2}: As also $\xi_{i}\in \ker P_{\gamma_{i-1}}(0)$,  we can see that
\begin{displaymath}
\begin{array}{rcl}
\displaystyle \frac{||X(n,0)\xi_{i}||}{||X(m,0)\xi_{i}||} & = & \displaystyle  \frac{||X(n,0)\gamma_{i-1}^{-n}\xi_{i}||}{||X(m,0)\gamma_{i-1}^{-m}\xi_{i}||}\,\gamma_{i-1}^{n-m} \\\\
& = &  \displaystyle  \frac{||X_{\gamma_{i-1}}(n,0)\xi_{i}||}{||X_{\gamma_{i-1}}(m,0)\xi_{i}||}\,\gamma_{i-1}^{n-m} \\\\\
& = &  \displaystyle  \frac{||X_{\gamma_{i-1}}(n,0)\xi_{i}||}{||X_{\gamma_{i-1}}(m,n)X_{\gamma_{i-1}}(n,0)\xi_{i}||}\,\gamma_{i-1}^{n-m} \\\\\
\end{array}
\end{displaymath}
and consequently
\begin{displaymath}
\begin{array}{rcl}
\left(\displaystyle \frac{||X(n,0)\xi_{i}||}{||X(m,0)\xi_{i}||}\right)^{-1} & = & \displaystyle \frac{||X_{\gamma_{i-1}}(m,n)X_{\gamma_{i-1}}(n,0)\xi_{i}||}{||X_{\gamma_{i-1}}(n,0)\xi_{i}||\gamma_{i-1}^{n-m}} \\\\
&=& \displaystyle \frac{||X_{\gamma_{i-1}}(m,n)[I-P_{\gamma_{i-1}}(n)]X_{\gamma_{i-1}}(n,0)\xi_{i}||}{||X_{\gamma_{i-1}}(n,0)\xi_{i}||\gamma_{i-1}^{n-m}} \\\\
&\leq & \displaystyle \frac{||X_{\gamma_{i-1}}(m,n)[I-P_{\gamma_{i-1}}(n)]||}{\gamma_{i-1}^{n-m}},
\end{array}
\end{displaymath}
where we have used that $X_{\gamma_{i-1}}$ is the transition matrix of (\ref{lineal-g}) with $\gamma=\gamma_{i-1}$ and the identity 
$$
[I-P_{\gamma_{i-1}}(n)]X_{\gamma_{i-1}}(n,0)\xi_{i}=X_{\gamma_{i-1}}(n,0)\xi_{i},
$$
which follows from (\ref{IP2}) and $\xi_{i}\in \ker P_{\gamma_{i-1}}(0)$.

Now, as $n\geq m$ and $\gamma_{i-1}\in \rho(A)$ it follows that
\begin{displaymath}
\begin{array}{rcl}
\left(\displaystyle \frac{||X(n,0)\xi_{i}||}{||X(m,0)\xi_{i}||}\right)^{-1} & \leq & \displaystyle  K_{\gamma_{i-1}}\left(\frac{\rho_{\gamma_{i-1}}}{\gamma_{i-1}}\right)^{n-m},
\end{array}
\end{displaymath}
and we can deduce that
\begin{displaymath}
\frac{1}{\overline{\beta}(\xi_{i})} < \frac{1}{\gamma_{i-1}} \quad \textnormal{with $\gamma_{i-1}\in \rho_{i-1}(A)$}.
\end{displaymath}

As the above result is also independent of the choice of $\gamma_{i-1}\in (b_{i-1},a_{i})$, we have that $a_{i}-\varepsilon <\overline{\beta}(\xi_{i})$ for any $\varepsilon>0$ small enough, which implies that 
\begin{equation}
\label{izquierda}
a_{i}\leq \overline{\beta}_{A}(\xi_{i}).
\end{equation}
\end{proof}

\begin{remark}
\label{obnp}
By following the lines of the proof of Theorem \ref{T1}, we can easily deduce that $\underline{\beta}_{A}(\xi_{i})\in [a_{i},b_{i}]$ for any initial condition $\xi_{i}\in \mathcal{W}_{i}(0)\setminus \{0\}$ of (\ref{lineal}).
\end{remark}

The above Remark combined with the inequality $\underline{\beta}_{A}(\xi_{i})\leq \overline{\beta}_{A}(\xi_{i})$ imply the following result:
\begin{corollary}
For any initial condition $\xi_{i}\in \mathcal{W}_{i}(0)\setminus \{0\}$ of (\ref{lineal}) with $i=1,\ldots,\ell$, it follows that $[\underline{\beta}_{A}(\xi_{i}),\overline{\beta}_{A}(\xi_{i})]\subseteq [a_{i},b_{i}]$.
\end{corollary}

\begin{theorem}
\label{T2}
For any $\xi\in \mathbb{R}^{d}\setminus\{0\}$ it follows that
\begin{displaymath}
[\underline{\beta}_{A}(\xi),\overline{\beta}_{A}(\xi)] \subseteq [a_{1},b_{\ell}].
\end{displaymath}
\end{theorem}

\begin{proof}
Let $\gamma_{0}\in \rho_{0}(A)$ and $\gamma_{\ell}\in \rho_{\ell}(A)$, then by statement (iii) of Remark \ref{obs1}, we have that $P_{\gamma_{0}}(0)=0$
and $P_{\gamma_{\ell}}(0)=I$. Then we have that $\xi \in \ima P_{\gamma_{\ell}}(0)$ and $\xi \in \ker P_{\gamma_{0}}(0)$.

As $\xi \in  \ima P_{\gamma_{\ell}}(0)$ we can proceed as in step 1 of the proof of Theorem \ref{T1} with $i=\ell$ and we will obtain that
$\overline{\beta}_{A}(\xi)\leq b_{\ell}$. Similarly; as $\xi \in \ker P_{\gamma_{0}}(0)$; we can proceed as in step 2 of the proof of Theorem \ref{T1} with $i=1$ and we will obtain that
$\overline{\beta}_{A}(\xi)\geq a_{1}$ and the property $\overline{\beta}_{A}(\xi) \in [a_{1},b_{\ell}]$  is verified. Finally, by following the previous lines we can deduce 
that $\underline{\beta}_{A}(\xi) \in [a_{1},b_{\ell}]$ and the result follows.
\end{proof}

Theorems \ref{T1} and \ref{T2} prompt the following question:
\begin{itemize}
    \item[\textbf{(Q)}] Can the extrema of either $[a_{i},b_{i}]$ or $\Sigma(A)$ be lower and upper Bohl's  exponents of a pair of initial conditions of (\ref{lineal})?.
\end{itemize}

As pointed out in \cite[pag.424]{Potzsche}, the answer is always affirmative when we consider the exponential dichotomy spectrum on the half--line; nevertheless; the answer seems more elusive when considering the 
whole axis. In order to explain this problem, we need to recall that the following definition:
\begin{definition}\cite{Gohberg,Nieza-15}
\label{KS}
The linear system (\ref{lineal}) is kinematically similar to 
\begin{equation}
\label{triangular}
y(n+1)=U(n)y(n)
\end{equation}
if there exists an invertible transformation $F\colon \mathbb{Z}\to M_{d}(\mathbb{R})$ with 
$$
\sup\limits_{n\in \mathbb{Z}}\{||F(n)||,||F^{-1}(n)||\}<+\infty
$$
such that the change of variables $y_{n}=F^{-1}(n)x_{n}$ leads to (\ref{triangular}).
\end{definition}

In addition, it is easy to prove that (see \emph{e.g.} \cite[p.183]{Siegmund}) the linear system (\ref{lineal}) is kinematically similar to an upper triangular system. From now on, we will assume that the system (\ref{triangular}) is upper triangular.

Contrarily to the autonomous case, the spectrum $\Sigma(U)$ of an upper triangular system does not always coincides with the spectrum of its diagonal  
coefficients, which prompts to consider the following property: 
\begin{definition}\cite{Potzsche}
The upper triangular system (\ref{triangular}) is diagonally significant if
$$
\Sigma(U)=\bigcup\limits_{i=1}^{d}\Sigma(u_{ii}),
$$
where $\Sigma(u_{ii})$ are the exponential dichotomy spectra of the scalar equations:
\begin{equation}
\label{escalares}
y_{i}(n+1)=u_{ii}(n)y_{i}(n) \quad \textnormal{for $i=1,\ldots,d$}.
\end{equation}
\end{definition}

The property of diagonal significance is always verified when considering the spectrum $\Sigma(U)$ on the half--line. Surprisingly enough; this is not the case
when considering the spectrum on $\mathbb{Z}$, which arises several difficulties when working on \textbf{(Q)}.

Finally, we will show some examples where \textbf{(Q)} has an affirmative answer.

\begin{example}
If (\ref{lineal}) is kinematically similar to an upper triangular system (\ref{triangular}) and:
\begin{itemize}
\item[i)] Its diagonal terms verifies 
\begin{equation}
\label{cotitas-0}
0<\inf\limits_{n\in \mathbb{Z}}|u_{ii}(n)| \leq \sup\limits_{n\in \mathbb{Z}}|u_{ii}(n)|<+\infty \quad \textnormal{for any $i=1\ldots,d$},
\end{equation}
\item[ii)] The system (\ref{triangular}) is diagonally significant.
\end{itemize}
Then for any spectral interval $[a_{i},b_{i}]$ there exist $\eta,\eta'\in \mathbb{R}^{d}\setminus\{0\}$ such that
$a_{i}=\underline{\beta}_{U}(\eta)$ and $b_{i}=\overline{\beta}_{U}(\eta')$. 
\end{example}

Indeed, the property (ii) combined with the fact that the exponential dichotomy spectrum is invariant by ki\-ne\-ma\-tical si\-mi\-larity (see \emph{e.g}, \cite[Cor.2.2]{Siegmund}) yield $\Sigma(A)=\Sigma(U)=\Sigma(u_{11})\cup \cdots \cup \Sigma(u_{dd})$.

By using Theorem 1 from \cite{Babiarz15} for $d=1$ (see also \cite[p.427]{Potzsche}), we know that (\ref{escalares}) has a unique lower and a unique upper Bohl's exponent. They are defined respectively as:
$$
\underline{\beta}(u_{ii})=\liminf\limits_{n-m\to +\infty}\prod\limits_{k=m}^{n-1}|u_{ii}(k)| \quad \textnormal{and} \quad
\overline{\beta}(u_{ii})=\limsup\limits_{n-m\to +\infty}\prod\limits_{k=m}^{n-1}|u_{ii}(k)|,
$$
which are independent of any initial condition. In addition, by using (\ref{cotitas-0}) combined with Proposition 2.4 from \cite{Potzsche} we have that
$$
\Sigma(u_{ii})=[\underline{\beta}(u_{ii}),\overline{\beta}(u_{ii})] \quad \textnormal{for any $i=1,\ldots,d$},
$$
and we conclude that any spectral interval $[a_{i},b_{i}]$ is a finite union of closed intervals whose boundary is composed by lower and upper Bohl's exponents of (\ref{triangular}).

The second example uses the fact that diagonal systems trivially satisfy the property of diagonal significance:
\begin{example}
If (\ref{lineal}) is kinematically similar to a diagonal system $y(n+1)=Dy(n)$ described by 
\begin{equation}
\label{diagonal}
y_{i}(n+1)=d_{i}(n)y_{i}(n)  \quad  \textnormal{with}  \quad  i=1,\ldots,d,
\end{equation}
where the diagonal terms verify (\ref{cotitas-0}),
then for any spectral interval $[a_{i},b_{i}]$ there exist $\eta,\eta'\in \mathbb{R}^{d}\setminus\{0\}$ such that
$a_{i}=\underline{\beta}_{D}(\eta)$ and $b_{i}=\overline{\beta}_{D}(\eta')$.
\end{example}

The property of diagonal significance plays a key role in the above examples. As we pointed out, this property  is trivially verified when considering the exponential dichotomy spectrum on the half--line. Nevertheless, the
diagonal significance is not always verified on the whole line and some sufficient conditions ensuring it are presented in \cite[Sect.4]{Potzsche}. To answer \textbf{(Q)} when diagonal significance is not verified remains as an open question.

\begin{remark}
As we were finishing this note, we realized that Barreira \emph{et al.} \cite[Th.7]{Barreira} studied a sequence of noninvertible bounded linear operators acting on a Banach space and developed a 
spectral theory based in the non--uniform exponential dichotomy. The authors obtained results related to our Theorems \ref{T1} and \ref{T2} considering Lyapunov exponents instead Bohl's ones. Nevertheless, 
this last fact combined with the use of other dichotomy property induced some technical differences between both approaches. 
\end{remark}

\end{document}